\documentclass[runningheads]{llncs}
\usepackage[T1]{fontenc}
\usepackage{graphicx, amsmath, hyperref}
\usepackage{color}

\spnewtheorem*{subclaim}{Subclaim}{\itshape}{\itshape}

\begin{document}
\title{Location-Domination Type Problems Under the Mycielski Construction\thanks{This work was sponsored by a public grant overseen by the French National Research Agency as part of the “Investissements d’Avenir” through the IMobS3 Laboratory of Excellence (ANR-10-LABX-0016) and the IDEX-ISITE initiative CAP 20-25 (ANR-16-IDEX-0001).}}
%
%\titlerunning{Abbreviated paper title}
% If the paper title is too long for the running head, you can set
% an abbreviated paper title here
%
\author{Silvia M. Bianchi\inst{1}%\orcidID{0000-1111-2222-3333}
\and
Dipayan Chakraborty\inst{2,3}\orcidID{0000-0001-7169-7288}
\and
Yanina Lucarini\inst{1}%\orcidID{2222--3333-4444-5555}
\and
Annegret K. Wagler\inst{2}}
\authorrunning{S. Bianchi et al.}
% First names are abbreviated in the running head.
% If there are more than two authors, 'et al.' is used.
%
\institute{Deprtamento de Matem\'{a}tica, Universidad Nacional de Rosario,\\
Av. Pellegrini 250, S2000BTP Rosario, Argentina\\
\email{\{sbianchi, lucarini\}@fceia.unr.edu.ar}
\and
Université Clermont-Auvergne, CNRS, Mines de Saint-Étienne, Clermont-Auvergne-INP, LIMOS, 63000 Clermont-Ferrand, France\\
\email{\{dipayan.chakraborty, annegret.wagler\}@uca.fr}
\and
Department of Pure and Applied Mathematics, University of Johannesburg, Auckland Park, 2006, South Africa}
\maketitle              % typeset the header of the contribution
\begin{abstract}
We consider the following variants of the classical minimum dominating set problem in graphs: locating-dominating set, locating total-dominating set and open locating-dominating set. All these problems are known to be hard for general graphs. A typical line of attack, therefore, is to either determine the minimum cardinalities of such sets in general or to establish bounds on these minimum cardinalities in special graph classes. In this paper, we study the minimum cardinalities of these variants of the dominating set under a graph operation defined by Mycielski in~\cite{Mycielski1955} and is called the Mycielski construction. %, applied to paths and cycles. 
We provide some general lower and upper bounds on the minimum sizes of the studied sets under the Mycielski construction. We apply the Mycielski construction to stars, paths and cycles in particular, and provide lower and upper bounds on the minimum cardinalities of such sets in these graph classes. Our results either improve or attain the general known upper bounds. 

\end{abstract}
\section{Introduction}

For a graph modeling a facility, the placement of monitoring devices, for example, fire detectors or surveillance cameras, motivates the study of various location-domination type problems in graphs. 
The problem of placing monitoring devices so that every site of a facility is visible from a monitor leads to a domination problem. In addition, the position of a fire, a thief, or a saboteur in the facility can be uniquely located by a specific subset of the monitoring 
devices which leads to location problems. 
During the last decades, several combined location-domination problems of this type have been actively studied, see, for example, the bibliography maintained by Lobstein and Jean~\cite{Lobstein_Lib}. 
In this work, we study three %%%five 
different location-domination type problems under a graph operation known as the Mycielski construction and defined by Mycielski himself in \cite{Mycielski1955}.

All graphs in this paper are finite, simple and connected. Given a graph $G=(V,E)$, the (open) neighborhood of a vertex $u\in V$ is the set $N(u) = N_G(u)$ of all vertices of $G$ adjacent to $u$, and $N[u] =N_G[u] = \{u\} \cup N(u)$ is the closed neighborhood of $u$. A subset $C \subseteq V$ is \emph{dominating} (respectively, \emph{total-dominating}) if the set $N[u]\cap C$ (respectively, $N(u)\cap C$) is non-empty for all $u \in V$. In addition, a subset $C \subseteq V$  \emph{separates} (respectively, \emph{total-separates}) a pair $u,v\in V$ if $N[u]\cap C\neq N[v]\cap C$ (respectively, $N(u)\cap C\neq N(v)\cap C$). In such a case, we also say that $u,v\in V$ are \emph{separated by} $C$ (respectively, \emph{total-separated} by $C$). A subset $C \subseteq V$ is called
\begin{itemize}
\item a \emph{locating-dominating set}~\cite{slater1} (or $LD$-set for short) of $G$ if it is a dominating set  of $G$ that separates all pairs of distinct vertices outside of $C$, that is, $N(u)\cap C\neq N(v)\cap C$, for all distinct $u,v\in V-C$;
\item a \emph{locating total-dominating set}~\cite{hhh} (or $LTD$-set for short) of $G$ if it is a total-dominating set of $G$ that separates all pairs of distinct vertices outside of $C$, that is, $N(u)\cap C\neq N(v)\cap C$, for all distinct $u,v\in V-C$;
\item an \emph{open locating-dominating set}~\cite{ss} (or $OLD$-set for short) of $G$ if it is a total-dominating set of $G$ that total-separates all pairs of distinct vertices of the graph, that is, $N(u)\cap C\neq N(v)\cap C$, for all distinct $u,v\in V$.
\end{itemize} 

Two distinct vertices $u,v$ of a graph $G=(V,E)$ are called \emph{false twins} if $N(u)=N(v)$, see \cite{ss}. Similarly, any two vertices $u,v\in V$ with $N[u]=N[v]$ are called \emph{true twins}. Now, for $X \in \{LD,LTD,OLD\}$, the $X$-problem on $G$ is the problem of finding an $X$-set of minimum size in $G$. The size of such a set is called the $X$-number of $G$ and is denoted by $\gamma_X(G)$. Note that a graph $G$ without isolated vertices admits an $OLD$-set if there are no false twins in $G$. On the other hand, LD-sets and LTD-sets are admitted by all graphs.

From the definitions themselves, the following relations hold for any graph $G$ admitting any two $X$-sets for $X \in \{LD, LTD, OLD\}$:
\begin{equation}\label{eq1}
\gamma_{LD}(G)\leq \gamma_{LTD}(G)\leq \gamma_{OLD}(G).  
\end{equation}

It has been shown that determining $\gamma_{X}(G)$ is in general NP-hard for all $X\in \{%%%DD,DTD,
LD,LTD,OLD\}$.  
%%%By \cite{CHL_2003}, determining $\gamma_{DD}(G)$ is NP-hard in general. It even remains hard for several graph classes where other in general hard problems are easy to solve, including bipartite graphs \cite{CHL_2003} and two classes of chordal graphs, namely split graphs and interval graphs \cite{F_2013}. 
%%%The identifying code problem has been actively studied during the last decade, where typical lines of attack are to determine minimum identifying codes of special graphs or to provide bounds for their size. 
%%%Closed formulas for the exact value of $\gamma_{DD}(G)$ have been found so far only for restricted graph families (e.g. for paths and cycles \cite{BCHL_2004}, for stars \cite{GM_2007}, for complete multipartite graphs \cite{DAM} and some subclasses of split graphs \cite{LNCS}). A linear time algorithm to determine $\gamma_{DD}(G)$ if $G$ is a tree was provided by Auger \cite{A_2010} and extended to block graphs in \cite{ABLW_2020,algo}.
%%%Also 
Apart from determining $\gamma_{LD}(G)$ being NP-hard in general~\cite{ss}, it remains so for bipartite graphs \cite{CHL_2003} and some subclasses of chordal graphs like split graphs and interval graphs~\cite{F_2017}. This result is also extended to planar bipartite unit disk graphs in \cite{MS_2009} and intersection graphs in \cite{F_2015}. Closed formulas for the exact values of $\gamma_{LD}(G)$ have so far been found for restricted graph families, for example, for paths \cite{slater1}, cycles \cite{BCHL_2004}, stars, complete multipartite graphs, some subclasses of split graphs and thin suns \cite{ABLW_2022,ABW_2015}. Bounds for the $LD$-number of trees were provided in \cite{BCMMS_2007}. A linear-time algorithm to determine $\gamma_{LD}(G)$ for $G$ being a tree was provided by Slater in \cite{slater2} and has been extended to block graphs (graphs which generalize the concept of trees in that any 2-connected subgraph in a block graph is complete) in~\cite{ABLW_2020}. Moreover, in connection to block graphs and hence, trees, tight upper and lower bounds for $LD$-numbers of block graphs and twin-free block graphs have been established in~\cite{CFPW_2023}.

Determining $\gamma_{OLD}(G)$ is NP-hard not only in general~\cite{ss} but also on other graph classes like perfect elimination bipartite graphs~\cite{P_2015}, interval graphs~\cite{F_2017} and is APX-hard on chordal graphs of maximum degree 4~\cite{P_2015}. Closed formulas for the exact value of $\gamma_{OLD}(G)$ have so far been found only for restricted graph families such as cliques and paths \cite{ss}, some subclasses of split graphs and thin suns \cite{ABLW_2022}. Tight lower and upper bounds for $OLD$-numbers certain classes of graphs like trees~\cite{ss}, block graphs~\cite{CFPW_2023}, lower bounds for interval graphs, permutation graphs and cographs~\cite{F_2015} and upper bounds for cubic graphs~\cite{HY_2014} have been established. Lastly, some algorithmic aspects of the problem have been discussed in \cite{ABLW_2020,P_2015}. 

Concerning $LTD$-sets, it can be checked that it is as hard as the $OLD$-problem by using the same arguments as in \cite{ss}. Bounds for the $LTD$-number of trees are given in \cite{hhh,LTD1}. In addition, the $LTD$-number in special families of graphs, including cubic graphs, grid graphs, complete multipartite graphs, some subclasses of split graphs and thin suns is investigated in \cite{ABLW_2022,LTD1}.

In fact, giving bounds for the $X$-numbers in special graphs is a popular way to tackle the problems. 
In this work, we study the behavior of the three $X$-sets of graphs under the following graph operation defined by Mycielski in~\cite{Mycielski1955}. Given a graph $G=(V,E)$ with $V=\{v_1,\dots,v_n\}$, a new graph $M(G)$ is constructed as follows: for every vertex $v_i$ of $G$, add a new vertex $u_i$ and make $u_i$ adjacent to all vertices in $N_G(v_i)$. Finally add a vertex $u$ which is adjacent to all $u_i$. Let the set containing all the vertices $u_i$'s be called $U$, that is, $U = \{u_1, u_2, \ldots , u_n\}$.

Originally, Mycielski introduced this construction in the context of graph coloring and used it to generate graphs $M(G)$ whose chromatic number increases by one compared to the chromatic number of $G$. In \cite{Mycie}, it is proved that the application of the Mycielski construction also increases the dominating number by one. 
In this paper, we show that the same holds for total domination and study the $X$-numbers of the graphs $M(G)$, where $G$ is a star $K_{1,n}$, a path $P_n$ and a cycle $C_n$ (see Fig. \ref{fig_M(star)}, Fig. \ref{fig_M(path)} and Fig. \ref{fig_M(cycle)}, respectively, for examples of their illustrations). As far as previous works on such variants of the dominating sets of Mycielski constructions is concerned, we know of only one such, namely, in~\cite{SVM_2022} where the authors find tight upper bounds of $ID$-numbers of $M(G)$ for $G$ being an identifiable graph (that is, a graph without true twins). The $ID$-number of an identifiable graph $G$ is the minimum cardinality of a dominating set $C$ of $G$ such that $N[u] \cap C \neq N[v] \cap C$ for all distinct pairs $u,v$ of vertices of $G$ (see~\cite{kcl}).

\begin{figure}[h]
\begin{center}

\includegraphics[scale=1.2]{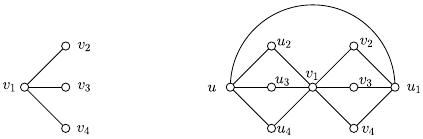}
\caption{The star $K_{1,3}$ and the resulting graph $M(K_{1,3})$}
\label{fig_M(star)}
\end{center}
\end{figure}

\begin{figure}[h]
\begin{center}
\includegraphics[scale=1.1]{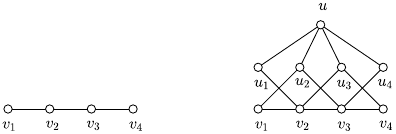}
\caption{The path $P_4$ and the resulting graph $M(P_4)$}
\label{fig_M(path)}
\end{center}
\end{figure}

\begin{figure}[h]
\begin{center}
\includegraphics[scale=0.9]{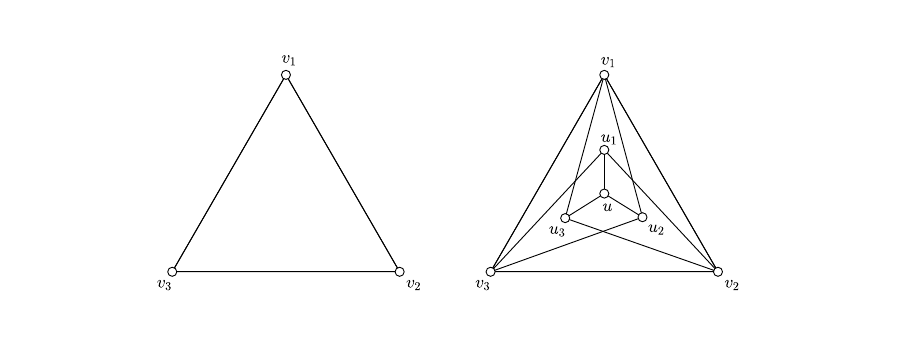}
\caption{The cycle $C_3$ and the resulting graph $M(C_3)$}
\label{fig_M(cycle)}
\end{center}
\end{figure}

In Section \ref{sec_lower-bounds}, 
we show that the application of the Mycielski construction increases the total-dominating number by at least one and give a general lower bound on the studied $X$-numbers of the graphs $M(G)$ in terms of $\gamma_{X}(G)$ when $G$ is either a path or a cycle. We then combine this bound with previously known results on $\gamma_{X}(P_n)$ (respectively, on $\gamma_{X}(C_n)$) to obtain lower bounds on the $X$-numbers of $M(P_n)$ (respectively, of $M(C_n)$). 

In Section \ref{sec_upper-bounds}, we give a general upper bound on the $X$-numbers of the graphs $M(G)$. We also show that this bound is attained when $G$ is a star and improve the bound for the cases when $G$ is a path or cycle.  

We note that there are some particularities in applying the Mycielski construction to paths and cycles with a small number of vertices. In fact, we have $M(P_2)=C_5$. While we have $\gamma_{LD}(P_2)=1$ and $\gamma_{LTD}(P_2)=\gamma_{OLD}(P_2)=2$, it is easy to see that $\gamma_{LD}(C_5)=2$, $\gamma_{LTD}(C_5)=3$, and $\gamma_{OLD}(C_5)=4$ hold. Moreover, $P_3$ and $C_4$ have false twins, and so do $M(P_3)$ and $M(C_4)$. Hence, there exist no $OLD$-sets of these graphs. However, we have $\gamma_{X}(P_3) = \gamma_{X}(C_4) = 2$ and $\gamma_{X}(M(P_3)) = \gamma_{X}(M(C_4)) = 4$ for $X\in \{LD,LTD\}$. Hence, in the rest of what follows, we study paths $P_n$ and cycles $C_n$ with larger values of $n$.

We close with some concluding remarks and open problems for future research.

\section{Lower bounds on $X$-numbers of graphs $M(G)$}
\label{sec_lower-bounds}

To start with, observe the following fact that, for every graph $G$,
\begin{enumerate}
\item two vertices $v_i$ and $v_j$ are false twins in $G$ if and only if the vertices $v_i,v_j$ and $u_i,u_j$ are pairs of false twins in  $M(G)$; and
\item $M(G)$ has no true twins.
\end{enumerate}
%\end{remark}

In \cite{Mycie}, it is proved that, for every graph $G$, the equality $\gamma(M(G))=\gamma(G)+1$ holds, where $\gamma(G)$ is the dominating number of $G$. 
Analogously, for the total-dominating number $\gamma_t(M(G))$, we can prove:

\begin{lemma}\label{lem_lowerBound}
For every graph $G$ without isolated vertices, we have $\gamma_{t} (M(G)) = \gamma_{t} (G)+1$. 
\end{lemma}

\begin{proof}[sketch]
Let $C\subseteq V$ be a total-dominating set of $G$ and let $u_i\in   U$. We define $C_i=C\cup \{u_i\}$. As $G$ has no isolated vertices, every vertex in $V(M(G))$ is adjacent to a vertex in $C_i$ and so, $\gamma_{t} (M(G))\leq \gamma_{t} (G)+1$. 

Now, let $C_M$ be a total-dominating set of $M(G)$ of cardinality $\gamma_{t} (M(G))$. Every vertex in  $V$ is adjacent to a vertex in $C_M$. Let us define the sets $C_V= C_M \cap V$ and $C_U=C_M \cap U$. Then it can be verified that the set $C_V\cup \{v_i:u_i\in C_U\}$ is a total-dominating set of $G$ and $|C_V| \leq |C_M - \{u\}|$. Thus, if $u \in C_M$, we are done. Therefore, let us assume that $u \notin C_M$. Then, there exists $u_j \in C_M$ in order for $C_M$ to total-dominate $u$. Now, for every vertex $u_i \in C_U$, any neighbor $v_k ~(\in V)$ of $u_i$ also has a neighbor in $C_V$ (the same vertex in $C_V$ that is a neighbor of $u_k$). This implies tat $C_M - C_U$ is a total-dominating set of $V$. Since, $u_j \in C_U$, we have $|C_U| \geq 1$ and hence, the result follows. \qed
\end{proof}

%With the help of these facts, we can establish the following general lower bound on $X$-numbers of graphs $M(G)$ in terms of $\gamma_{X}(G)$:

This motivates us to study the parameter $\gamma_X(M(G))$ in terms of $\gamma_X(G)$. In doing so, we now establish a general lower bound on the $X$-numbers of the graphs $M(G)$, where $G$ is either a path $P_n$ or a cycle $C_n$ and $X \in \{LD, LTD, OLD\}$. 

\begin{theorem}\label{thm_general-lb}
Let $X\in \{LD,LTD,OLD\}$. For a graph $G$ that is either a path $P_n$ or a cycle $C_n$ admitting an $X$-set, we have 
\[
\gamma_{X}(M(G))\geq \gamma_{X}(G)+1.
\]
\end{theorem}

\begin{proof}[sketch]
As a proof sketch, we provide here the proof of the theorem only for the case that $X = LD$. The proof in the other cases when $X \in \{LTD,OLD\}$ follows with similar proof techniques. To begin with, let us assume that $G$ is any graph (not necessarily a path or a cycle) and that $C_M$ is a minimum $LD$-set of $M(G)$. Let $C_V= C_M \cap V$ and $C_U=C_M \cap U$. Then define the set $C = C_V \cup \{v_i : u_i \in C_U\}$.

\medskip

\noindent \textit{Claim.} $C$ is an $LD$-set of $G$.

\medskip

\noindent \textit{Proof of Claim.}
Firstly, since $C_M$ is a dominating set of $M(G)$, one can verify that the set $C$ is also a dominating set of $G$. We now show that $C$ is also a separating set of $G$. Let $v_i, v_j \in V$ be any pair of arbitrary vertices such that $v_i, v_j \notin C$. Then we show that $v_i, v_j$ are separated by $C$ in $G$. Let $w_k \in C_M$ with $k \notin \{i,j\}$ separate $v_i$ and $v_j$ in $M(G)$, where $w_k \in \{ u_k, v_k \}$. If $w_k=v_k$, then $C$ clearly separates $v_i, v_j$. So, let $w_k = u_k$ and that $u_k$ is a neighbour of $v_i$ and not of $v_j$ in $M(G)$. Now, if $k=j$, then $v_j \in C$ and is trivially separated from every other vertex of $G$ by the definitions of $LD$--sets. So, let $k \neq j$. Then again, $v_k$ is a neighbour of $v_i$ and not of $v_j$ in $G$ and thus $C$ separates $v_i,v_j$. This establishes the claim. \qed

Thus we have,
\begin{flalign}\label{eq_thm-lb}
\gamma_{LD} (M(G)) \geq |C_M - \{u\}| \geq |C| \geq \gamma_{LD}(G).
\end{flalign}
So, if $u \in C_M$, then $\gamma_{LD}(M(G)) > |C_M|$ and hence, the statement of the theorem holds. So, let us assume that $u \notin C_M$, in which case, we have $\gamma_{LD} (M(G)) \geq \gamma_{LD}(G)$. Toward contradiction, let us assume that $\gamma_{LD} (M(G)) = \gamma_{LD}(G)$. Then, by the assumed equalities in (\ref{eq_thm-lb}), $C$ must be a minimum $LD$-set of $G$. This in turn implies that for each $i$, we have $|\{u_i,v_i\} \cap C_M| \leq 1$. In other words, if $u_i \in C_M$, then $v_i \notin C_M$ and vice-versa.

For the rest of the proof sketch, let us assume that $G$ is either a path $P_n$ or a cycle $C_n$. First of all, we observe that if any three consecutive vertices $v_i, v_{i+1}, v_{i+2} \in C$, then $C$ cannot be a minimum $LD$-set of $G$, as one can discard $v_{i+1}$ from $C$ and the latter still remains an $LD$-set. Similarly,  for some $i$, if $v_i, v_{i+1}, v_{i+3}, v_{i+4} \in C$, then again $C$ cannot be a minimum $LD$-set of $G$, as one can discard $v_{i+1}, v_{i+3}$ and include $v_{i+2}$ in $C$ and the latter still remains an $LD$-set of $G$. With those observations, let us first assume that some vertex $u_i \in C_M$ in order to dominate $u$. If one of its neighbours in $G$, say $v_{i+1}$, without loss of generality, belongs to $C_M$, then we must also have $v_{i+2} \in C_M$ in order for $C_M$ to dominate $u_{i+1}$ (note that $v_i \notin C_M$). Thus, $v_i, v_{i+1}, v_{i+2} \in C$, a contradiction to the minimality of $C$ by our earlier observation. Hence, let us assume that for all $u_i \in C_M$, none of its neighbours in $G$, that is, $v_{i-1}$ and $v_{i+1}$, belong to $C_M$. So, fix one such $u_i \in C_M$. Then $v_i \notin C_M$. Therefore, without loss of generality, let $u_{i+1} \in C_M$ in order for the latter to dominate $v_i$. If any of $u_{i-1}, v_{i-1} \in C_M$, then again we would have three consecutive vertices of $G$ in $C$, a contradiction. So, let us assume that $u_{i-1}, v_{i-1} \notin C_M$. In order for $v_{i-1}, v_{i+1}$ to be separated, we must have either $w_{i-2} \in C_M$ or $w_{i+2} \in C_M$, where $w_{i-2} \in \{u_{i-2}, v_{i-2}\}$ and $w_{i+2} \in \{u_{i+2}, v_{i+2}\}$. However, we cannot have $w_{i+2} \in C_M$, as otherwise, we would have $v_i, v_{i+1}, v_{i+2} \in C$, the same contradiction as before. Hence, $w_{i-2} \in C_M$. If $w_{i-2} = v_{i-2}$, then $u_{i-2} \notin C_M$. This implies that $v_{i-3} \in C_M$ for $C_M$ to dominate $u_{i-2}$. This implies that $v_{i-3}, v_{i-2}, v_i, v_{i+1} \in C$, a contradiction by our earlier observation. Moreover, if $w_{i-2} = u_{i-2}$, then $v_{i-2} \notin C$ and hence, $w_{i-3} \in C_M$ for $v_{i-2}$ to be dominated by $C_M$, where $w_{i-3} \in \{u_{i-3}, v_{i-3}\}$. Here again, we have $v_{i-3}, v_{i-2}, v_i, v_{i+1} \in C$, the same contradiction.

This proves that our assumption of $\gamma_{LD}(M(G)) = \gamma_{LD}(G)$ is wrong which, in turn, proves the theorem for the case that $X = LD$. The other cases when $X \in \{LTD, OLD\}$ follow with similar proof techniques. \qed 
\end{proof}

Note that this lower bound in Theorem \ref{thm_general-lb} is tight: 
\begin{itemize}
%%%\item for $X\in \{DD,DTD\}$, we have $\gamma_X(P_4)=3$ as $\{v_1,v_2,v_3\}$ is a minimum $X$-set, and $\gamma_X(M(P_4))=4$ as $\{v_1,v_2,u_3,u\}$ is a minimum $X$-set (see Figure \ref{fig_M(path)} for $P_4$ and $M(P_4)$), 
\item for $X\in \{LD,LTD\}$, we have  
$\gamma_X(C_3)=2$ as $\{v_1,v_2\}$ is a minimum $X$-set, and 
$\gamma_X(M(C_3))=3$ as $\{v_1,v_2,u\}$ is a minimum $X$-set 
(see Figure \ref{fig_M(cycle)} for $C_3$ and $M(C_3)$), 
%\item \edit{change!!!}
\item for $X = OLD$, no tight examples are yet known in this case.
%for $X=OLD$, we have 
%$\gamma_{OLD}(P_4)=4$ as $\{v_1,v_2,v_3,v_4\}$ is a minimum $X$-set and $\gamma_{OLD}(M(P_4))$ $=5$ as $\{v_1,u_2,u_3,v_4,u\}$ is a minimum $OLD$-set \remove{change}.
\end{itemize}

%\begin{figure}[h]
%\begin{center}
%\includegraphics[scale=1.1]{W5.pdf}
%\caption{The graph $W_5$ (the black vertices form a minimum $OLD$-set)}
%\label{OLDcota+1}
%\end{center}
%\end{figure}

We deduce lower bounds for $\gamma_X(M(P_n))$ and $\gamma_X(M(C_n))$ from the respective values of $\gamma_X(P_n)$ and $\gamma_X(C_n)$. Theorem \ref{thm_general-lb} together with the results 
from \cite{BCHL_2004} on $\gamma_{LD}(P_n)$ and $\gamma_{LD}(C_n)$ yield:

\begin{corollary}\label{cor_lb_LD}
If $G$ equals $P_n$ or $C_n$ for $n \geq 3$, we have as lower bound:
$$
\gamma_{LD}(M(G)) \geq \left\lceil\frac{2n}{5}\right\rceil +1.
$$
\end{corollary}

The exact $OLD$-numbers of path and cycles are studied in \cite{ss} and \cite{ABLW_2019}, respectively. However, the latter result for cycles of even order needed to be corrected and as such, we state and prove the result in its entirety as follows.

\begin{theorem} \label{thm_old_cycles}
For any cycle $C_n$ on $n$ vertices such that $n \geq 3$ and $n \neq 4$, 
we have
$$\gamma_{OLD}(C_n) = \left\{
\begin{array}{rll} 
  &\left \lceil\frac{2n}{3} \right\rceil, &\text{for odd } n, \vspace{2pt} \\ 
  2 &\left \lceil\frac{n}{3} \right\rceil, &\text{for even } n.
\end{array}
\right.$$
\end{theorem}

\begin{proof}
We prove the theorem by first showing that $\left \lceil\frac{2n}{3} \right\rceil$ for odd $n$ and $2\left \lceil\frac{n}{3} \right\rceil$ for even $n$ is a lower bound on $\gamma_{OLD}(C_n)$ and then providing an $OLD$-set of $C_n$ of exactly the same cardinality as the lower bound. We start with establishing the lower bound first.

Seo and Slater showed in~\cite{ss} that if $G$ is a regular graph on $n$ vertices, of regular-degree~$r$ and with no open twins, then we have $\gamma_{OLD}(G) \geq \frac{2}{1+r}n$. Using this %%%last 
result in~\cite{ss} for the cycle $C_n$, therefore, we have $\gamma_{OLD}(C_n) \geq \frac{2}{3}n$, that is, $\gamma_{OLD}(C_n) \geq \big \lceil \frac{2}{3}n \big \rceil$. Now, for $n \neq 6k+4$ for any non-negative integer $k $, we have
$$\bigg \lceil \frac{2n}{3} \bigg \rceil = \left\{
\begin{array}{rll} 
  &\left \lceil\frac{2n}{3} \right\rceil, &\text{for odd } n, \vspace{2pt} \\ 
  2 &\left \lceil\frac{n}{3} \right\rceil, &\text{for even } n.
\end{array}
\right.$$
Thus, the only case left to prove is the following claim.

\begin{claim}
For $n=6k+4$ with $k \geq 1$, we have $\gamma_{OLD}(C_n) \geq 2 \left \lceil\frac{n}{3} \right\rceil = 4k+4$.
\end{claim}

\begin{proof}[of Claim]
The proof of the last claim is by induction on $k$ with the base case being for $k=1$, that is, when $C_n$ is a cycle on $n=10$ vertices. We fisrt show the result for $n=10$.

\begin{subclaim}
$\gamma_{OLD}(C_{10}) \geq 8$.
\end{subclaim}

\begin{proof}[of Subclaim]
Let $V(C_{10}) = \{v_1, v_2, \ldots , v_{10}\}$ and $S$ be a minimum open $OLD$-set of $C_{10}$. Then, $|S| < 10$ by a charaterization result by Foucaud et al.~\cite{FGRS_2021}  on the extremal graphs $G$ for which $\gamma_{OLD}(G) = |V(G)|$. Hence, there exists a vertex $v_1$ (without loss of generality) such that $v_1 \notin S$. We consider the induced $5$-paths $P_1 : v_1v_2v_3v_4v_5$ and $P_2: v_1v_{10}v_9v_8v_7$. Then, from the path $P_1$, the vertex $v_3 \in S$ for the latter to total-dominate $v_2$ and the vertex $v_5 \in S$ for the latter to separate the pair $v_2,v_4$. By the same argument, from path $P_2$, the vertices $v_9,v_7 \in S$. Moreover, at least one vertex from each of the pairs $(v_2,v_4), (v_4,v_6), (v_6,v_8), (v_8,v_{10}), (v_{10},v_2)$ must belong to $S$ for the latter to total-dominate $v_3, v_5, v_7, v_9, v_1$, respectively. Hence, the result follows from counting. \qed
\end{proof}

Thus, the result holds for the base case of the induction hypothesis. We, therefore, assume $k \geq 2$ and that $\gamma_{OLD}(C_m) \geq 4q+4$ for all cycles $C_m$ with $|V(C_m)| = 6q+4$ and $q \in \{1, 2, \ldots , k-1\}$. Toward contradiction, let us assume that $\gamma_{OLD}(C_n) < 4k+4$. Moreover, let $V(C_n) = \{v_1, v_2, \ldots , v_n\}$. Then again, by the charaterization result in~\cite{FGRS_2021}, we have $\gamma_{OLD}(C_n) < n$. This implies that, for any minimum $OLD$-set $S$ of $C_n$, there exists a pair $(v_{n-6},v_{n-5})$ (by a possible renaming of vertices) such that $v_{n-6} \in S$ and $v_{n-5} \notin S$. Let $C'_{n-6}$ be the cycle on $n-6$ vertices formed by adding the edge $v_1v_{n-6}$ in the graph $C_n - \{v_{n-5},v_{n-4}, \ldots, v_n\}$. Note that $|V(C'_{n-6})| = 6(k-1)+4$ and hence, the induction hypothesis applies to it to give
\begin{equation} \label{eq_old}
\gamma_{OLD}(C'_{n-6}) \geq 4(k-1)+4 = 4k.
\end{equation}
Now, let $S' = S - \{v_{n-5},v_{n-4}, \ldots, v_n\}$.

\begin{subclaim}
$S'$ is an $OLD$-set of $C'_{n-6}$.
\end{subclaim}

\begin{proof}[of Subclaim]
To show that $S'$, first of all, is a total-dominating set of $C'_{n-6}$, we notice that the vertices $v_1, v_5 \notin S$. Therefore, all vertices in the set $\{v_2, v_3, \ldots , v_{n-6}\}$ remain total-dominated by $S'$. Moreover, $S'$ total-dominates $v_1$ by virtue of $v_{n-6} \in S'$. This proves that $S'$ is a total-dominating set of $C'_{n-6}$.

We now show that $S'$ is also a total-separating set of $C'_{n-6}$. To that end, since $v_{n-5} \notin S$, if now the vertex $v_n \notin S$ as well, then $S'$ clearly total-separates every pair of vertices in $C'_{n-6}$ and hence, is an $OLD$-set. If however, $v_n \in S$ and total-separates a pair of vertices in $C'_{n-6}$, the pair can either be $(v_1,v_2)$ or $(v_1,v_3)$. Since $n \geq 16$, we have $ 3 < n-6$ and hence, $v_{n-6} \in S'$ total-separates the pairs $(v_1,v_2)$ and $(v_1,v_3)$ in $C'_{n-6}$. Therefore, $S'$ is an $OLD$-set of $C'_{n-6}$. \qed
\end{proof}

\begin{subclaim}
$|S \cap \{v_{n-5},v_{n-4}, \ldots, v_n\}| \geq 4$.
\end{subclaim}

\begin{proof}[of Subclaim]
Since $v_{n-6} \notin S$, it implies that $v_{n-3} \in S$ in order for the latter to total-dominate the vertex $v_{n-4}$. Moreover, we also have $v_{n-1} \in S$ in order for $S$ to total-separate the pair $(v_{n-2}, v_{n-4})$. Furthermore, we must have at least one vertex each from the pairs $(v_{n-4}, v_{n-2})$ and $(v_{n-2}, v_n)$ belonging to $S$ in order for the latter to total-dominate the vertices $v_{n-3}$ and $v_{n-1}$, respectively. Finally, at least one vertex from the pair $(v_{n-4}, v_n)$ must also belong to $S$ in order for $S$ to total-separate the pair $(v_{n-3}, v_{n-1})$. This proves that the result holds. \qed
\end{proof}

Recall that $|S| = \gamma_{OLD}(C_n) < 4k+4$, by assumption. Thus, we have
$$\gamma_{OLD}(C'_{n-6}) \leq |S'| = |S| - |S \cap \{v_{n-5},v_{n-4}, \ldots, v_n\}| <4k + 4 - 4 = 4k,$$
a contradiction to the Inequality (\ref{eq_old}). This proves the claim and establishes the lower bound on $\gamma_{OLD}(C_n)$. \qed
\end{proof}

The theorem is, therefore, proved by providing an $OLD$-set $S$ of $C_n$ of the exact same cardinality as the lower bound, that is,
\begin{eqnarray} \label{eq_thm_old}
|S| = \left\{
\begin{array}{rll} 
  &\left \lceil\frac{2n}{3} \right\rceil, &\text{for odd } n, \vspace{2pt} \\ 
  2 &\left \lceil\frac{n}{3} \right\rceil, &\text{for even } n.
\end{array}
\right.
\end{eqnarray}
Let = $V(C_n) = \{v_1, v_2, \ldots , v_n\}$ and that $n = 6k+r$, where $r \in \{0,1,2,3,4,5\}$. For $k=0$, that is, $C_n$ being either a $3$-cycle or a $5$-cycle, it can be checked that the sets $\{v_1,v_2\}$ and $\{v_1, v_2, v_3, v_4\}$ are the respective $OLD$-sets. Thus, the result holds in this case. For the rest of this proof, therefore, we assume that $n \geq 6$, that is, $k \geq 1$. We now construct a vertex subset $S$ of $C_n$ by including in $S$ the vertices
\begin{enumerate}
\item $v_{6i-4}, v_{6i-3}, v_{6i-1}, v_{6i}$ for all $i \in \{1,2, \ldots , k\}$ if $r=0,3$. In this case, we have $|S| = 4k$ for $r=0$ and $|S| = 4k+2$ for $r=3$.
\item $v_{6i-4}, v_{6i-3}, v_{6i-2}, v_{6i-1}$ for all $i \in \{1,2, \ldots , k\}$ if $r \neq 0$; with
\begin{enumerate}
\item the vertices $v_{6k}, v_{6k+1}, \ldots , v_{6k+r-1}$ if $r=1,2,4$. In this case, we have $|S| = 4k+r$; and
\item the vertices $v_{6k+1}, v_{6k+2}, v_{6k+3}, v_{6k+4}$ if $r=5$. In this case, we have $|S| = 4k+4$.
\end{enumerate}
\end{enumerate}

It can be checked that the constructed set $S$ is, indeed, an $OLD$-set of $C_n$ and of the cardinality as in Equation~(\ref{eq_thm_old}). This proves the result. \qed
\end{proof}
%
%$$
%\gamma_{OLD}(C_n) \geq \left\{
%\begin{array}{ll}
%  \left\lceil\frac{2n}{3} \right\rceil & \text{for odd } n \geq 5 \\
%  2 \left\lceil\frac{n}{3} \right\rceil & \text{for even } n \geq 6 \\
%\end{array}
%\right.
%$$
Combining Theorem \ref{thm_general-lb} with results on $\gamma_{OLD}(P_n)$ in~\cite{ss} and on $\gamma_{OLD}(C_n)$ in Theorem~\ref{thm_old_cycles}, we deduce:

\begin{corollary}\label{cor_lb_OLD}
Consider $P_n$ with $n =6k+r$ for $k \geq 1$ and $r \in \{0, \ldots,5\}$, then we have:
$$
\gamma_{OLD}(M(P_n)) \geq \left\{
\begin{array}{lcl}
  4k+r+1 & \text{if} & r \in \{0,\ldots,4\}, \\
  4k+5   & \text{if} & r=5; \\
\end{array}
\right.
$$
and for $n \geq 3$ and $n \neq 4$, we have
$$\gamma_{OLD}(M(C_n)) \geq \left\{
\begin{array}{rll} 
  &\left \lceil\frac{2n}{3} \right\rceil +1, &\text{for odd } n, \vspace{2pt} \\ 
  2 &\left \lceil\frac{n}{3} \right\rceil +1, &\text{for even } n.
\end{array}
\right.$$
%
%$$
%\gamma_{OLD}(M(C_n)) \geq \left\{
%\begin{array}{ll}
%  \left\lceil\frac{2n}{3} \right\rceil  +1& \text{for odd } n \geq 5 \\
%  2 \left\lceil\frac{n}{3}  \right\rceil+1 & \text{for even } n \geq 6 \\
%\end{array}
%\right.$$
\end{corollary}

Theorem \ref{thm_general-lb} together with the results 
from \cite{hhh} on $\gamma_{LTD}(P_n)$ and 
from \cite{LTD1} on $\gamma_{LTD}(C_n)$ imply:

\begin{corollary}\label{cor_lb_LTD}
If $G$ equals $P_n$ or $C_n$ for $n \geq 3$, we have as lower bound:
$$
\gamma_{LTD}(M(G)) \geq \left\lfloor \frac{n}{2}\right\rfloor - \left\lfloor \frac{n}{4}\right\rfloor + \left\lceil\frac{n}{4} \right\rceil +1.
$$
\end{corollary}

%\textcolor{blue}{}
%\textcolor{red}{}

\section{Upper bounds on $X$-numbers of graphs $M(G)$}
\label{sec_upper-bounds}

%We first establish two general upper bounds on $X$-numbers of graphs $M(G)$, in terms of $\gamma_{X}(G)$ and $|V(G)|$.
We first establish a general upper bound on $X$-numbers of graphs $M(G)$ in terms of $\gamma_{X}(G)$.

\begin{theorem}\label{wow}
Let $X\in \{%%%DD,DTD,
LD,LTD,OLD\}$. For a graph $G$ admitting an $X$-set, we have
\[
\gamma_{X}(M(G))\leq 2\gamma_{X}(G).
\]
\end{theorem}

\begin{proof}[sketch]
Let $C$ be a minimum $X$-set of $G$. Then, we construct a new set $C' = C \cup \{u_i : v_i \in C\}$. It can be checked that if $C$ is a dominating (respectively, total-dominating) set of $G$, then so is it of $M(G)$. For any $x \in V(M(G))$, let $N_M(x)$ (respectively, $N_M[x]$) denote the neighborhood (respectively, closed neighborhood) of $x$ in $M(G)$. If $C$ is a total-separating set of $G$, then for any $x \in V(M(G))$, we have
\begin{itemize}
\item $C' \cap N_M(u)=\{u_i:v_i\in C\}$

\item $C' \cap N_M(u_j)=(C\cup \{u_i:v_i\in C\})\cap N_M(u_j)=C\cap N(v_j)$

\item $C' \cap N_M(v_j)=(C \cup \{u_i:v_i\in C\}) \cap N_M(v_i)= \{v_k,u_k: v_k\in N(v_j)\cap C\}$
\end{itemize}

As is evident, the set $C' \cap N_M(x)$ is unique for each $x \in V(M(G))$. Thus, $C'$ is also a total-separating set of $M(G)$. Moreover, $|C'| = 2 |C|$. This proves the result. \qed
\end{proof}

Based on results from \cite{ABW_2015,GM_2007,LTD1} on $X$-numbers of stars and the relation (\ref{eq1}), 
%and the example shown in figure \ref{cotadoble}, 
we can show that the bound given in Theorem \ref{wow} is tight for stars (see Fig. \ref{cotadoble} for illustration):

\begin{theorem}\label{thm_X-stars}
For stars $K_{1,n}$ with $n \geq 3$, we have $\gamma_{X}(K_{1,n})=n$ and $$\gamma_{X}(M(K_{1,n}))=2n$$ whenever $X\in\{%%%DD,DTD,
LD,LTD\}$.
\end{theorem}

\begin{figure}[!t]
\begin{center}
\includegraphics[scale=0.9]{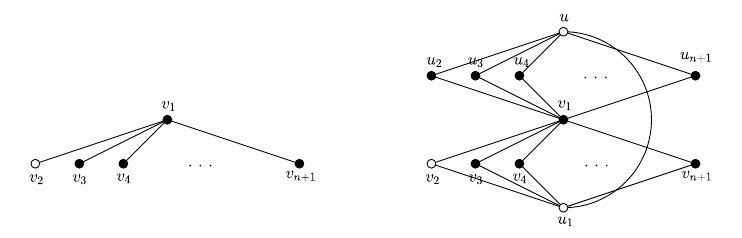}
\caption{$K_{1,n}$ and $M(K_{1,n})$ (black vertices form a minimum $X$-set when $X\in\{%%%DD,DTD,
LD,LTD\}$)}
\label{cotadoble}
\end{center}
\end{figure}

Note that stars $K_{1,n}$ have false twins and, therefore, so does $M(K_{1,n})$. Hence, $M(K_{1,n})$ does not admit an $OLD$-set. Fig. \ref{OLDcotadoble} provides an example for $\gamma_{X}(M(G)) = 2\gamma_{X}(G)$ when $X=OLD$. For $OLD$-sets, we can further prove the following.

\begin{figure}[!b]
\begin{center}
\includegraphics[scale=0.9]{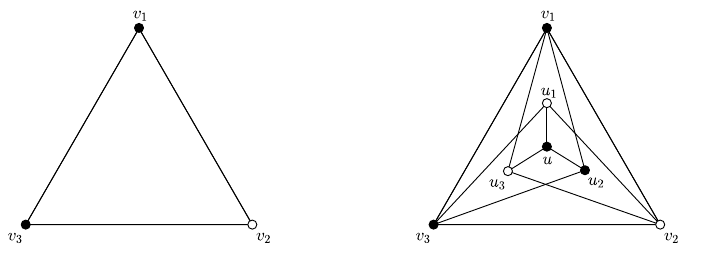}
\caption{$C_3$ and $M(C_3)$ (black vertices form a minimum $OLD$-set)}
\label{OLDcotadoble}
\end{center}
\end{figure}

\begin{theorem}\label{M(OLD)}
Let $G$ be a graph without isolated vertices and false twins. Then $\gamma_{OLD}(M(G))\leq \gamma_{OLD}(G)+2$.
\end{theorem}

\begin{proof}[sketch]
Let $C\subset V$ be an $OLD$-set of $G$ and let $u_i\in   U$. We define $C_i=C\cup \{u,u_i\}$. As $G$ has no isolated vertices, every vertex in $V(M(G))$ is adjacent to a vertex in $C_i$. This implies that $C_i$ is a total-dominating set of $M(G)$.
Moreover, by the fact that $C$ is a total-separating set of $G$, it can be checked that each of the following sets is unique. 
$$\begin{array}{ll}
C_i\cap N_M(u)=\{u_i\}; & \\
C_i\cap N_M(v_j)=(C\cap N(v_j))\cup \{u_i\} & \text{ for } v_j\in N(u_i);\\
C_i\cap N_M(v_j)=C\cap N(v_j) & \text{ for } v_j\notin N(u_i); \text{ and}\\
C_i\cap N_M(u_j)=(C\cap N(v_j))\cup \{u\} & \text{ for } u_j\in U.
\end{array}$$
This proves that $C_i$ is an $OLD$-set of $M(G)$ and since, $|C_i| = |C|+2$, the result follows. \qed
\end{proof}

The bound given in Theorem \ref{M(OLD)} is tight, as $M(P_2)$ and $M(C_3)$ (in Figure \ref{OLDcotadoble}) show. In addition, it enables us to prove the following for $\gamma_{OLD}(M(P_n))$ and $\gamma_{OLD}(M(C_n))$.

\begin{theorem} \label{thm_old_paths&cycles}
For all $n \geq 2$ and $n \neq 3$, we have $$\gamma_{OLD}(M(P_n)) = \gamma_{OLD}(P_n)+2$$ 
and for all $n \geq 3$, we have
$$\gamma_{OLD}(C_n)+1 \leq \gamma_{OLD}(M(C_n)) \leq \gamma_{OLD}(C_n)+2.$$
\end{theorem}

\begin{proof}[sketch]
The result for cycles follows directly from Theorems \ref{thm_general-lb} and \ref{M(OLD)}. For paths, again using Theorem~\ref{M(OLD)}, we only need to show that $\gamma_{OLD}(P_n) \geq \gamma_{OLD}(P_n)+2$ for all $n \geq 2$ and $n \neq 3$. As far as small paths a concerned, it can be checked that $\gamma_{OLD}(P_2)=2$, $\gamma_{OLD}(P_4) = \gamma_{OLD}(P_5) = 4$; and $\gamma_{OLD}(M(P_2)) = 4$, $\gamma_{OLD}(M(P_4)) = \gamma_{OLD}(M(P_5))=6$. Thus, the result holds for these small paths. Therefore, we assume that $n = 6k+r$ with $k \geq 1$, where $r \in \{0, 1, \ldots , 5\}$. If $V(P) = \{v_1, v_2, \ldots , v_n\}$, the proof follows by partitioning the vertex set of $M(P_n)$ into $\lceil \frac{n}{6} \rceil$ parts, the first $\lfloor \frac{n}{6} \rfloor$ of which are given by $B_i = \{v_j, u_j : 6i-5 \leq j \leq 6i \}$ for all $1 \leq i \leq \lfloor \frac{n}{6} \rfloor$; and the last (if exists, that is, if $r \neq 0$) part $B_l = \{v_j, u_j : 6k+1 \leq j \leq r \}$. Further analysis of any block $B_i$ for $1 \leq i \leq \lfloor \frac{n}{6} \rfloor$ shows that any $OLD$-set $C$ of $M(P_n)$ must contain at least $4$ vertices from $B_i$. Moreover, we would have $|C \cap U| \geq 2$. This gives the total count for the cardinality of $C$ to be $4k+2$ and thus proves the theorem for $r=0$. Moreover, each other case for $r \in \{1,2,3,4,5\}$ is dealt with separately where it can be shown that, for $r \in \{1,2,3,4\}$, exactly $r$ vertices and, for $r=5$, exactly $4$ vertices need to be included in $C$. This proves the theorem by comparison to the results for $\gamma_{OLD}(P_n)$ in~\cite{ss}. \qed
\end{proof}

Concerning $LD$-numbers, we note that $\gamma_{LD}(M(P_2))=2$ and $\gamma_{LD}(M(P_3))=\gamma_{LD}(M(P_4))=\gamma_{LD}(M(P_5))=4$ holds. 
We can improve the general upper bounds for $\gamma_{LD}(M(P_n))$ and $\gamma_{LD}(M(C_n))$ as follows:

\begin{theorem}\label{thm_LD_paths&cycles}
Consider $P_n$ with $n =3k+r$ for $k \geq 2$, $r \in \{0,1,2\}$ and $C_n$ with $n\geq 3$, then we have:
$$\gamma_{LD}(M(P_n)) \leq \left\{
\begin{array}{rcl}
	2k+1 & \text{if} & r=0 \\
  2k+2 & \text{if} & r \in \{1,2\} \\
\end{array}
\right.$$
and
$$\gamma_{LD}(M(C_n)) \leq \left\{
\begin{array}{lcl}
	n-\left\lfloor \frac{n}{3}\right\rfloor +1 & \text{if} & n \text{ is odd }  \\
  n-2\left\lfloor \frac{n}{6}\right\rfloor +1  & \text{if} &  n \text{ is even} \\
\end{array}
\right.$$
%Moreover, the bounds are tight.
\end{theorem}

\begin{proof}[sketch]
We provide the proof sktech for paths to illustrate the proof technique. The proof for cycles follows with similar techniques. Let $n \geq 6$, $n =3k+r$ with $k \geq 2$ and $r \in \{0,1,2\}$. Then, according to three possible values of $r$, we define the following sets. 
\begin{itemize}
\item If $r=0$, we define $C=\{v_2\} \cup \{v_{i},v_{i+1} : i= 3 \ell +1, \ell \in \{1, \ldots, k-1\}\} \cup \{u_{3k},u\}$. In this case, we have $|C|=2k+1$.
\item If $r=1$, we define $C_1=(C - \{u_{3k}\}) \cup \{v_{3k+1},u_{3k+1}\}$. Here we have, $|C_1|=2k+2$
\item If $r=2$, we define $C_2=(C-\{u_{6k}\}) \cup \{v_{6k+1},v_{6k+2}\}$. In this case, we have $|C_2|=2k+2$
\end{itemize}
Further analysis of the above sets $C$, $C_1$ and $C_2$ shows that in each of the above three cases, the sets are $LD$-sets of $M(P_n)$. The result then follows by the cardinalities of the sets in the above three cases. \qed
\end{proof}

We observe that the upper bounds are tight for $M(P_n)$ with $6 \leq n \leq 8$ and for $M(C_n)$ with $n \in \{3,6,7\}$, but are not tight for $M(C_4)$ and $M(C_5)$, for example. There are no examples yet known where the upper bounds are not tight for $M(P_n)$.  

The next theorem provides an upper bound for the $LTD$-numbers of $M(P_n)$ and $M(C_n)$. However, before coming to it, as far as small graphs of these graph classes are concerned, we note that $\gamma_{LTD}(M(P_2))=3$, $\gamma_{LTD}(M(P_3))=4$ and $\gamma_{LTD}(M(P_4))$ $=\gamma_{LTD}(M(P_5))=5$. The next result improves the general upper bounds for $\gamma_{LTD}(M(P_n))$ and $\gamma_{LTD}(M(C_n))$ as follows.

\begin{theorem} \label{thm_LTD_paths&cycles}
Consider $P_n$ with $n =6k+r$ for $k \geq 1$, $r \in \{0, \ldots,5\}$ and $C_n$ with $n\geq 3$, then we have:
$$\gamma_{LTD}(M(P_n)) \leq \left\{
\begin{array}{lcl}
	4k+2 & \text{if} & r=0 \\
  4k+r+1 & \text{if} & r \in \{1,2,3\} \\
	4k+r & \text{if} & r \in \{4,5\} \\
\end{array}
\right.$$
and
$$\gamma_{LTD}(M(C_n)) \leq \left\{
\begin{array}{lcl}
	n-\left\lfloor \frac{n}{3}\right\rfloor +2 & \text{if} & n \text{ is odd }  \\
  n-2\left\lfloor \frac{n}{6}\right\rfloor +2  & \text{if} &  n \text{ is even} \\
\end{array}
\right.$$
%Moreover, the bounds are tight.
\end{theorem}

\begin{proof}[sketch]
The upper bound on the $LTD$-number of $M(P_n)$ follows by the fact that $\gamma_{LTD}(M(P_n)) \leq \gamma_{OLD} (M(P_n)) = \gamma_{OLD} (P_n) + 2$ (by Theorem \ref{thm_old_paths&cycles}) and by the known exact values of $\gamma_{OLD} (P_n)$ from \cite{ss}.

For the upper bound on the $LTD$-number of $M(C_n)$, we consider the following two graphs $G$ and $G'$. Let $G=(V,E)$ be the graph such that $V=\{v_1,v_2,\ldots,v_n\}$ and $E=\{ \{v_{i},v_{i+2}\}, \{v_{i},v_{i+4}\}, \{v_{i},v_{i+3}\} : i \in \{1,\ldots,n\}\}$ (where the sum of the indices is taken modulo $n$). Renaming the vertices of $V$ in such a way that $w_i=v_{1+2i}$ for $i \in \{0,\ldots,n-1\}$, we consider the second graph $G'$ with vertex set $\{w_i: i \in \{0, \ldots, n-1\}\}$ and edge set $\{\{w_i,w_{i+1}\},\{w_i,w_{i+2}\}: i \in \{0, \ldots, n-1\}\}$. We then look at the graph $G'$ with vertex set $V(C_n)$, and denote by $C_{G'}$ its minimum vertex cover. Then, it is easy to check that $C_{G'}\cup \{u_i,u\}$ for some $i\in\{1,\ldots,n\}$ is an $LTD$-set of $M(C_n)$. The theorem for $C_n$ therefore follows by the use of another result proven separately that the size of a minimum vertex cover of $G$ is 
 
\begin{itemize}
\item  $n-\left\lfloor \frac{n}{3}\right\rfloor$ when $n$ odd and not a multiple of $3$,
\item  $n-\left\lfloor \frac{n}{4}\right\rfloor$ when $n$ is odd and multiple of $3$,
\item $n-2\left\lfloor \frac{n}{6}\right\rfloor$ when $n$ is even.
\end{itemize} \qed
\end{proof}

We observe that for $\gamma_{LTD}(M(C_n))$, there are values of $n$ where the upper bound is attained (for example, $n \in \{6,9\}$), but also where this is not the case (for example, $n \in \{3,4,5,7,8\}$). 
There are no examples yet known where the upper bounds are not tight for $M(P_n)$.

\section{Concluding remarks}

To summarize, we studied three location-domination type problems under the Mycielski construction. In Section \ref{sec_lower-bounds}, we showed that $\gamma_{X}(G)+1$ is a general lower bound of $\gamma_{X}(M(G))$ for all paths and cycles and all $X\in \{LD,LTD,OLD\}$. Using results on $\gamma_{X}(P_n)$ (respectively, on $\gamma_{X}(C_n)$) from \cite{ABLW_2019,BCHL_2004,LTD4,hhh,ss}, this allowed us to deduce the appropriate lower bounds on the $X$-numbers of $M(P_n)$ (respectively, of $M(C_n)$) for $X\in \{LD,LTD,OLD\}$. As a related extension of one of the main focuses of this paper, namely, the $OLD$-numbers of $M(C_n)$, we also establish the exact $OLD$-numbers for cycles.

In Section \ref{sec_upper-bounds}, we firstly provided two general upper bounds on $X$-numbers of the graphs $M(G)$. We showed that the upper bound of $2\gamma_{X}(G)$ is attained when $G$ is a star for $X\in \{LD,LTD\}$. For $OLD$-numbers of $M(G)$, we could further establish the general upper bound of $\gamma_{OLD}(G)+2$. We showed that this bound is attained for $\gamma_{OLD}(M(P_n))$ and, combining our results on the lower and upper bound of the $OLD$-numbers, we obtained a Vizing-type result for $M(C_n)$, namely, $\gamma_{OLD}(C_n)+1 \leq \gamma_{OLD}(M(C_n)) \leq \gamma_{OLD}(C_n)+2$. For the other $X$-problems with $X\in \{%DD,DTD,
LD,LTD\}$, we could improve the general upper bounds for the $X$-numbers of both $M(P_n)$ and $M(C_n)$.  

For the studied $X$-numbers, there are examples where the upper bounds are attained (and, therefore, cannot be improved any further). On the other hand, there are also examples where the upper bounds are not tight. 
Therefore, our future research includes finding these exact values. In view of the fact that lower bounds were obtained by considering the domination aspect only, we expect that the true values are closer to the upper bounds. This applies particularly to $\gamma_{LD}(M(P_n))$ and to $\gamma_{LTD}(M(P_n))$ where no examples are yet known where the upper bound is not tight. 

Moreover, it would be interesting to study similar questions for other locating-dominating type problems, for example, differentiating total-dominating sets (defined as total-dominating sets that separate all vertices of the graph).

%\begin{credits}
%\subsubsection{\ackname} We sincerely thank the reviewers and acknowledge their comments and suggestions toward guiding this article to its current form.
%\end{credits}
%
% ---- Bibliography ----
%
% BibTeX users should specify bibliography style 'splncs04'.
% References will then be sorted and formatted in the correct style.
%
% \bibliographystyle{splncs04}
% \bibliography{mybibliography}
%

\end{document}